\documentclass[12pt]{amsart}

\usepackage{graphicx}
\usepackage{psfrag}

\usepackage{color}

\usepackage{caption}
\usepackage{amsmath}
\usepackage{amsthm}
\usepackage{amssymb}
\usepackage{amscd}

\newtheorem{thm}{Theorem}[section]
\newtheorem{prop}[thm]{Proposition}
\newtheorem{lem}[thm]{Lemma}

\newtheorem*{thmA}{Theorem A}
\newtheorem*{thmB}{Theorem B}
\newtheorem*{thmBb}{Theorem C}
\newtheorem{corollary}[thm]{Corollary}

\newtheorem*{thmC}{Theorem D}
\newtheorem{addendum}[thm]{Addendum} 

\newcommand{\refthmB}{B}
\newcommand{\refthmBb}{C}

\theoremstyle{definition}
\newtheorem{remark}[thm]{Remark}

\newcommand{\iso}{\cong}
\renewcommand{\L}{\mathcal{L}}
\newcommand{\A}{\mathcal{A}}

\newcommand{\D}{\mathbb{D}}
\newcommand{\R}{\mathcal{R}}
\newcommand{\Z}{\mathbb{Z}}
\newcommand{\N}{\mathbb{N}}
\newcommand{\G}{\mathcal{G}}
\newcommand{\Q}{\mathcal{Q}}
\renewcommand{\P}{\mathcal{P}}

\newcommand{\e}{\varepsilon}

\newcommand{\area}{{\rm{area}}}

\newcommand{\<}{\langle}
\renewcommand{\>}{\rangle}
\renewcommand{\|}{\mid}

\newcommand{\emf}[1]{\textbf{\textit{#1}}} 

\begin{document}

\title[Algorithmic construction of classifying spaces]{On
the algorithmic construction of classifying spaces and
the isomorphism problem for biautomatic groups}

\author{Martin R.~Bridson}
\address{Martin R.~Bridson\\
Mathematical Institute \\
24--29 St Giles'\\
Oxford OX1 3LB  \\ 
U.K. }
\email{bridson@maths.ox.ac.uk}
 
\author{Lawrence Reeves}
\address{Lawrence Reeves\\ Department of Mathematics and Statistics\\
University of Melbourne\\
Parkville VIC 3010 \\
Australia
}
\email{lreeves@unimelb.edu.au} 
\thanks{This research was supported by the EPSRC of Great Britain.}



\begin{abstract}
We show that the isomorphism problem is solvable in the class of
central extensions of word-hyperbolic groups, and that the
isomorphism problem for biautomatic groups reduces to that for
biautomatic groups with finite centre. We describe an algorithm
that, given an arbitrary finite presentation of an automatic
group $\Gamma$, will construct explicit finite models for the skeleta 
of $K(\Gamma,1)$ and hence compute the integral homology and cohomology of $\Gamma$.
\end{abstract}
\maketitle

\centerline{\em{For Fabrizio Catanese on his 60th birthday}}

\medskip

There are several natural classes of finitely presented groups that cluster around the
notion of non-positive curvature, ranging from hyperbolic groups to combable groups (see
\cite{mrb:icm} for a survey and references).
 The isomorphism problem is solvable in the class of
hyperbolic groups 
but is unsolvable in the class of combable groups \cite{mrb:isoC}.
It remains unknown whether the isomorphism is solvable in the intermediate classes, such as (bi)automatic groups
and {\rm{CAT}}$(0)$ groups. 
Hyperbolic groups also form one of the very few interesting classes in which there is an algorithm
that, given a finite presentation of a group $\Gamma$ in the class, will construct finite models for the skeleta
of a $K(\Gamma,1)$. For finitely presented groups in general, one cannot even calculate $H^2(\Gamma,\Z)$; see 
\cite{CG}. Our focus in this article will be on the isomorphism problem
for biautomatic groups and the construction problem for classifying spaces of combable and automatic
groups. 

We remind the reader that the isomorphism problem for a class $\mathcal G$ of
finitely presented groups is said to be solvable if there exists
an algorithm that takes as input pairs of finite group
presentations $(P_1,P_2)$ and proceeding under the
assumption that the groups $|P_i|$ belong to $\mathcal C$,
decides whether or not
$|P_1|\iso|P_2|$.
The first purpose of this article is to point out that the
isomorphism problem for biautomatic groups (or any subclass of
such groups) can be reduced to the problem of determining
isomorphism of the groups modulo their centres. We write $Z(G)$ to
denote the centre of a group $G$.

\begin{thmA}
Let $\G$ be a class of biautomatic groups. Let $\Q$ be the class
of groups $\{G/Z(G) \mid G\in\G\}$. If the isomorphism problem is
solvable in $\Q$ then it is solvable in $\G$.
\end{thmA}

Zlil Sela \cite{S}  proved that the isomorphism problem is
solvable among finite presentations of freely-indecomposable, torsion-free hyperbolic
groups, and  his work has recently been extended by 
Fran{\c{c}}ois Dahmani and Vincent Guirardel to cover all hyperbolic
groups \cite{DGuir}; see also  \cite{DGroves}.
Walter Neumann and Lawrence Reeves \cite{NR} 
proved that central extensions of hyperbolic groups are
biautomatic. Thus we have:
\begin{corollary}
The isomorphism problem is solvable in the class of central extensions of
hyperbolic groups.
\end{corollary}

In Section \ref{s:makeKG1} we describe an algorithm for constructing the
skeleta of classifying spaces for combable groups (given an explicit fellow-traveller
constant). From this we deduce:

\begin{thmB}\label{thmB}
There is an algorithm that, given a positive integer $d$ and any finite presentation of an
automatic group $\Gamma$, will construct  a compact $(d+1)$-skeleton for $K(\Gamma,1)$,
i.e.~an explicit finite, connected cell complex $K$ with $\pi_1K\cong\Gamma$ and $\pi_iK=0$
for $2\le i \le d$.
\end{thmB}

\begin{thmBb}\label{thmBb}
There is an algorithm that, given any finite presentation of an
automatic group $\Gamma$, will calculate $H_*(\Gamma,A)$ and
$H^*(\Gamma,A)$, where $A$ is any finitely 
generated abelian group and the action of $\Gamma$ on $A$
is trivial.
\end{thmBb}

The algorithm for calculating $H^2\Gamma$ furnishes the following
major ingredient for  the proof of Theorem A.

\begin{corollary}
There is an algorithm that computes
a complete irredundant list of central
extensions of any given automatic group 
by a given finitely generated abelian
group.
\end{corollary}

Lee Mosher proved that central quotients of biautomatic groups are
biautomatic \cite{M}. By combining this with Theorems A and C, and
some well-known facts about subgroups of biautomatic groups, we
shall prove:

\begin{thmC}
The isomorphism problem among biautomatic groups is
solvable if and only if the isomorphism problem is solvable
among biautomatic groups with finite centre.
\end{thmC}

It follows from Theorems C and D that if the isomorphism problem
for biautomatic groups is unsolvable, then there must exist a
recursive sequence of  finite presentations such that each
of the groups presented is biautomatic, all of the groups in
the sequence have finite centre and isomorphic integral homology
and cohomology groups, but one cannot tell which of the
groups presented are isomorphic.

\section{Determining the centre of a biautomatic group}\label{s:1}

The main purpose of this section is to prove the following:

\begin{prop}\label{centre} There exists an algorithm that takes as input
an arbitrary finite presentation of a biautomatic group and
which gives as output the isomorphism type of the centre of the
group, and a finite set of words that generate the centre.
\end{prop}

We shall assume that the reader is familiar with the basic
vocabulary of automatic group theory, as set out in the seminal
text \cite{E+}. It is convenient  to fix the following notation.

Let $\Gamma$ be a group with finite generating set $\A$. The free
monoid on $\A$ is denoted $\A^*$, and the natural surjection
$\A^*\to \Gamma$ is denoted $\mu$. We assume that $\A$ is equipped
with an involution, written $a\mapsto a^{-1}$ such that
$\mu(a^{-1}) = \mu(a)^{-1}$.

Given a language $\L\subseteq\A^*$, we define the language
$\L^{-1}\subseteq\A^*$ to be the set of formal inverses of $\L$,
that is, $a_1\cdots a_n\in\L^{-1}$ if and only if $a_n^{-1}\cdots
a_1^{-1}\in\L$. A language $\L\subset\A^*$ is called a
\emf{biautomatic structure} for $\Gamma$ if the restriction of
$\mu : \A^*\twoheadrightarrow\Gamma$  to each of $\L$ and
$\L^{-1}$ is an automatic structure --- see \cite[Definition
2.5.4]{E+}.

We remind the reader that associated to a  biautomatic structure
$\L\twoheadrightarrow\Gamma$ one has a \emf{fellow-traveller}
constant $k>0$ with the property  that for letters $a,a'\in\A^{\pm
1}$, words $w,w'\in\L$ with $\mu(aw)=\mu(w'a')$ and positive
integers $t$, one has $d(\mu(aw_t),\mu(w'_t))\le k$, where $d$ is
the word metric associated to $\A^*\twoheadrightarrow\Gamma$ and
$u_t$ denotes the prefix of length $t$ in the word $u$.

In \cite[Chapter 5]{E+} an algorithm is described which, given a
finite presentation for an automatic group, will produce an
automatic structure for the group. Although not explicitly stated
in \cite{E+}, the following generalization is
a straightforward consequence of their
proof.

\begin{prop}\label{find}
There exists an algorithm that, given a finite presentation of a
biautomatic group, will construct a biautomatic structure for the
group and calculate a fellow-traveller constant for that
structure.
\end{prop}

\begin{proof}
We describe the changes needed to the algorithm
in section 5.2 of
\cite{E+}.

Given a finite state automaton $W$ with accepted language $\L$
(over the given generating set), the algorithm given in \cite{E+}
will, if $\L$ is part of an automatic structure, eventually
terminate and give a full description of the automatic structure
(i.e., all multiplier automata and the equality checker). Let
$W^{-1}$ be a finite state automaton with accepted language
$\L^{-1}$. Applying this algorithm to both $W$ and $W^{-1}$ gives
an algorithm which will terminate if $\L$ is a biautomatic
structure.

Using a standard diagonal argument, this procedure is applied
`simultaneously' to all pairs of finite state automata
$(W,W^{-1})$ with the given input alphabet. The fact that the
group is biautomatic ensures that at some point the algorithm
terminates. One can then obtain a fellow-traveller constant
directly from the geometry of the automata (cf. Lemma 2.3.2
\cite{E+}).
\end{proof}

\begin{lem}\label{list}
There is an algorithm that, given a finite presentation of an
automatic group, will list all of the words that represent the
identity in order of increasing length.
\end{lem}

\begin{proof}
As above, we first calculate an automatic structure for the group.
Let $\A$ denote the generating set of the presentation. Given an
enumeration of all words over $\A$ in order of increasing length, we can
use the equality checker to test whether each is equal in the
group to the identity (which is represented by the empty word).
\end{proof}

\begin{lem}
There exists an algorithm that, given a biautomatic structure
($\L\twoheadrightarrow\Gamma$) for a group $\Gamma$, will
calculate the sublanguage of $\L$ that evaluates onto the centre
of $\Gamma$ (i.e., the language $\L\cap\mu^{-1}(Z(\Gamma))$).
\end{lem}

\begin{proof}
Denote by $C(g)$ the centralizer of a group element $g$. Setting
$\L_a=\L\cap\mu^{-1}(C(\mu(a)))$, we have $\L\cap
\mu^{-1}(Z(\Gamma))=\cap_{a\in\A}\L_a$.  If the fellow-traveller
constant of the biautomatic structure is $k$, then by definition
$d(1, \mu(p^{-1}ap)) \le k$ for all prefixes $p$ of  words
$w\in\L$ such that $\mu(wa)=\mu(aw)$. Thus, writing $\P^k_a$ for
the language of words $w\in\A^*$ such that $d(1, \mu(p^{-1}ap))
\le k$ for all prefixes $p$ of $w$, we see that $\L_a\subseteq
\P^k_a$. Therefore, since the intersection of finitely many
regular languages is regular, it suffices to construct a finite
state automaton over $\A$ with accepted language $\P^k_a\cap
\mu^{-1}(C(\mu(a)))$.

The set of  states of the desired automaton is the ball $B(k,1)$
of radius $k$ about $1\in\Gamma$ in the word metric, together
with one other (fail) state $\phi$. The state corresponding to
$\mu(a)$ is both the start and (unique) accept state. The
transitions are given by
\begin{align*}
B(k,1)\ni g&\overset{b\in\A}\longmapsto
\begin{cases}
\mu(b)^{-1}g\mu(b)&\text{if}\ \mu(b)^{-1}g\mu(b)\in B(k,1)\\
\phi&\text{otherwise}
\end{cases}\\
\phi&\overset{b\in\A}\longmapsto\phi\\
\end{align*}
That is, if the machine is in state $g$ when it reads the letter
$b$ from the input tape, then it moves to the fail state if
conjugation by $\mu(b)$ sends $g$ to an element outside the ball
$B(k,1)$, and it moves to $\mu(b)^{-1}g\mu(b)$ if it is in the
ball.
\end{proof}

\begin{prop}
There is an algorithm that, given a finite presentation $\langle
\A\mid \R \rangle$ for a biautomatic group $\Gamma$, will calculate
a finite set of words in $\A^*$ that generates the centre of
$\Gamma$, and will give a finite presentation of $Z(\Gamma)$ in
terms of these generators.
\end{prop}

\begin{proof}
Proposition \ref{find} yields an explicit biautomatic structure
for $\Gamma$, together with a fellow-traveller constant $k>0$ for
that structure. In the course of the  preceding proof we
implicitly showed that $Z(\Gamma)$ is generated by words from
$\A^*$ that have length at most $2k+1$. Thus, in order to obtain
an explicit set of generators for $Z(\Gamma)$ we need only check
which words of length at most $2k+1$ commute with all of the
generators $\A$ of $\Gamma$. And this one can do by listing all of
the words of length at most $4k+4$ that represent the identity in
$\Gamma$, using Lemma \ref{list}.

In fact, the preceding proof shows that $Z(\Gamma)$ is a
quasiconvex subgroup of $\Gamma$ (a result originally due to
Gersten and Short \cite{GS}) with a quasiconvexity constant $K$
that can be calculated from $k$ (see \cite{AB} pages 94--95). It
follows  that one obtains a finite
presentation for $Z(\Gamma)$ by simply calculating which
concatenations of strings of generators of $Z(\Gamma)$, with total
length  $2(2k+1)+2(2K+[K(2k+1)]$ in $\A^*$, represent $1\in
\Gamma$.
\end{proof}

We can now complete the proof of Proposition \ref{centre}. Given a
presentation of a biautomatic group, we calculate a finite
presentation of $Z(\Gamma)$ as above. The isomorphism problem for
abelian groups is solvable and one can make an explicit list $(P_n)$ of
finite presentations, exactly one for each isomorphism type of
finitely generated abelian group. One then looks for an
isomorphism between  $Z(\Gamma)$ and the groups on this
list by simply enumerating homomorphisms from  $Z(\Gamma)$ to each of the
groups and {\em{vice versa}}, looking for an inverse pair,  as in Lemma \ref{partisoalg}.
In more detail: for each of the presentations on the list $(P_n)$, the construction 
of Lemma \ref{partisoalg} provides
a partial algorithm that will successfully terminate if the  group $G_n=|P_n|$  is isomorphic
to $Z(\Gamma)$; the algorithm that we are describing here runs in a diagonal manner ---
first it runs one step of the procedure that looks for an isomorphism between $Z(\Gamma)$ and $G_1$,
then one step of the procedure comparing $Z(\Gamma)$ to $G_2$, then a further  two steps comparing
$Z(\Gamma)$ to $G_1$,  to $G_2$, and to $G_3$; then a further three steps of the  
procedures comparing $Z(\Gamma)$ to  $G_1$,  to $G_2$, to $G_3$, and to $G_4$, 
and so on.

This completes the proof of Proposition \ref{centre}.

\begin{remark}\label{sec1}
In deciding the isomorphism class of $Z(\Gamma)$ above, we
appealed to the general solution for the isomorphism problem for
finitely generated abelian groups. In the final section we present
some results which provide a more efficient search that exploits
the rational structure on $Z(\Gamma)$.
\end{remark}

In the proof of Proposition \ref{centre}
we used the following general and well-known result, which we make explicit
for the sake of clarity. 

\begin{lem} \label{partisoalg}
There is a partial algorithm that,
given two finite presentations, will search for an
isomorphism between the two groups: if the presentations define
isomorphic groups then this procedure  will eventually halt;
if the groups are not isomorphic then it will not terminate.
\end{lem}

\begin{proof} Given finite presentations $ \< A_1\mid R_1\>$ and
$\<A_2\mid R_2\>$ defining groups $G_1$ and $G_2$, one can enumerate all 
homomorphisms from $G_1$ to $G_2$ by running through all choices of 
words $\{u_a \mid a\in A_1\}$ in the free group on $A_2$, treating these
as putative images of the generators of $G_1$ in $G_2$: one freely reduces
the words obtained by substituting $u_a$ for each occurrence of $a$ in the
relations $R_1$ -- call these words $\{\rho_r\mid r\in R_1\}$;
one  tries to verify that
$a\mapsto u_a$ defines a homomorphism $G_1\to G_2$ by listing
all products of conjugates of the relations $R_2^{\pm 1}$, freely reducing them
and comparing the freely-reduced form to the words $\rho_r$; the assignment
$a\mapsto u_a$ defines a homomorphism $G_1\to G_2$ if and only if this
procedure eventually produces all of the words $\rho_r$.

One applies the same process with the indices reversed  to enumerate all
homomorphisms from $G_2$ to $G_1$. In parallel, one tests all pairs of
homomorphisms $f_1:G_1\to G_2$ and $f_2:G_2\to G_1$ that are found in order to
see if they are mutually inverse. Again this test is carried out by a naive search:
the homomorphisms are described by explicit formulae saying where they send the
generators, so to check that $f_2\circ f_1  = {\rm{id}}_{G_1}$, for example,
one simply has to check if a list of words $(w_a : a\in A_1)$ defines the same
indexed set of elements of $G_1$ as $(a : a\in A_1)$; one is interested only
in a positive answer, so one does not need a solution to the word problem for
this, one just enumerates all products of conjugates of the relations $R_1^{\pm 1}$,
checking to see if each is {\em{freely}} equal to $a^{-1}w_a$.
\end{proof}

\section{A reduction of Theorem A}\label{s:reduce}

Given a group $G$ and an abelian group $A$, a group $E$ is called
a central extension of $G$ by $A$ if there is a short exact
sequence
$$0\to A\to E\to G\to 1$$
and the map $G\to \operatorname{Aut}(A)$ induced by conjugation in
$E$ is trivial. We remind the reader that such central extensions
are classified up to equivalence by the cohomology class $[z]\in
H^2(G,A)$ of the cocycle $z: G\times G\to A$ that is defined by choosing a
set-theoretic section $s:G\to E$ of the given surjection and
setting $z(g,g')= s(gg')s(g')^{-1}s(g)^{-1}$ (see \cite[Section
IV.3]{B}). 

Suppose now that $\G$ and $\Q$ are as in the statement of Theorem
A, and that we are given finite presentations
$\langle\A_1\mid\R_1\rangle$ and $\langle\A_2\mid\R_2\rangle$ for
groups $G_1$, \mbox{$G_2\in \G$}. Denote by $Z_1$  the centre of
$G_1$ and denote by $Q_1$  the quotient $G_1/Z_1$. Define $Z_2$
and $Q_2$ similarly. The results of the previous section and the
hypothesis that the isomorphism problem is solvable in $\Q$ allow
us to decide the isomorphism types of the groups
$Q_1,Q_2,Z_1,Z_2$. If $Q_1\not\iso Q_2$  or $Z_1\not\iso Z_2$,
then we conclude that the original groups $\Gamma_1$ and
$\Gamma_2$ are not isomorphic. Thus Theorem A has been reduced to
a problem of deciding whether two central extensions are equivalent.

Suppose now that $Q_1\iso Q_2$ and $Z_1\iso Z_2$, and refer to
these groups as $Q$ and $Z$ respectively. 
In Section \ref{s:KG1} we will construct an
irredundant enumeration of the possible central extensions of $Q$
by $Z$. In the light of 
the following observation, this enumeration allows us to
determine whether $G_1$ and $G_2$ are isomorphic.

\begin{lem}\label{iso}
Let  $\G$ be a class of finitely presented groups.
Given an
irredundant enumeration of (presentations for) the groups in $\G$,
one can decide whether or not an arbitrary pair of finite presentations
of groups from $\G$ define isomorphic groups.
\end{lem}
\begin{proof}
Given two finite presentations  $G_1=\< \A_1\mid \R_1\>$ and
$G_2=\< \A_2\mid \R_2\>$ with $G_1,G_2\in\G$, one sets $G_2$
aside and searches the enumeration of $\G$ to identify which group
on the list is isomorphic to $G_1$. One does this by using Lemma \ref{partisoalg}
repeatedly, as in the proof of Proposition \ref{centre}. 
Repeating this procedure with $G_2$ in place of $G_1$
 will determine whether or not $G_1$
and $G_2$ are isomorphic to the same element in the enumeration of
$\G$ and hence to one another.
\end{proof}

\section{Algorithmic construction of 
classifying spaces}\label{s:KG1}\label{s:makeKG1}

The considerations in the previous section compel us to
enumerate the possible central extensions of $Q$ by
$Z$, and for this we need an algorithm
 to calculate $H^2(Q, Z)$ starting from
any finite presentation of $Q$. More generally,  we wish to calculate $H^*(Q, A)$
and $H_*(Q, A)$, where $A$ is a finitely generated abelian
group. We shall achieve this by describing an
algorithm that constructs finite skeleta of
a $K(Q,1)$. This construction (Theorem \ref{real}) depends
on something less than the existence of a combing of $Q$ and
explicit knowledge of the fellow-traveller constant
for this combing; in the case of automatic groups, this
constant can be calculated as in Lemma 2.3.2 of \cite{E+}
(cf. Proposition 1.1). The
construction is similar to that described by
S.~M.~Gersten \cite{Ger} in proving that asynchronously automatic groups
are of type ${\rm{F}}_3$. See also
\cite{alonso} and  \cite[section 10.2]{E+}.

\medskip
 
Each attaching map in our $K(Q,1)$ will be given
by a subdivision map followed by a restricted form of
singular combinatorial map. By definition,
if  $f:L\to K$ is a
\emf{singular combinatorial} map
 between CW-complexes then for every open $n$-cell
$\sigma\in L$, either the restriction of $f$ to $\sigma$ is a
homeomorphism onto an open $n$-cell 
of $K$, or else $f(\sigma)\subset K^{(n-1)}$. The more restrictive
notion of a \emf{semi-combinatorial map}  is obtained by 
requiring that if $f(\sigma)\subset K^{(n-1)}$ then 
$f|_\sigma$ is a constant map.  
A semi-combinatorial complex is
one where the attaching maps of all cells are semi-combinatorial.

The $K(Q,1)$ that we will construct is
 a CW-complex in which many $n$-cells are standard $n$-cubes,
combinatorially. The attaching maps of the remaining cells are
defined on subdivisions of a restricted type on the boundary of
a standard cube; the attaching maps themselves are a restricted
kind of semi-combinatorial map.

\subsection{$k$-Lipschitz contractions}

We remind the reader that a group $G$ with finite generating set
$\A=\A^{-1}$ is said to be \emf{combable} if there is a
constant $k$ and a (not necessarily regular) sublanguage
$\{\sigma_g: g\in G\} \subset\A^*$ mapping bijectively to $G$
under the homomorphism $\A^*\to G$ such that $\rho(\sigma_g(t),
\sigma_{ga}(t))\le k$ for all $g\in G$,  $a\in\A$ and all integers
$t>0$, where $\rho$ is the word metric associated to $\A$ and 
$\sigma_g(t)$ is the image in $G$ of the prefix of
length $t$ in $\sigma_g$ (this prefix is taken to be equal to the
whole word if $t$ is greater than the length of the word).
Such a constant $k$ is called a \emf{fellow-traveller constant}.

One says that a finitely generated group $G$ {\emf{admits  $k$-Lipschitz contractions}}
if, given every finite subset $S\subset G$, there is a map  $H_S:S\times\mathbb N\to G$ such that for all $s,t\in S$ and $n\in\N$
we have
 $\rho(H_S(s,n),H_S(t,n))\le k\, \rho(s,t)$ 
and $\rho(H_S(s,n),H_S(s,n+1))\le 1$, with
 $H_S$  constant on $S\times \{0\}$ and $H_S(\ast,n)={\rm{id}}_S$ for $n$ sufficiently large (where $\rho$ denotes
the word metric).

If $G$ is combable with combing $\sigma_g$ and fellow-traveller constant $k$,
then one obtains  $k$-Lipschitz contractions  by defining $H_S(s,n)=\sigma_s(n)$, regardless of $S$.
More generally, groups that admit a coning of finite (asynchronous) width in the sense of 
\cite{mrb:geom} admit $k$-Lipschitz contractions.

A well-known argument that has appeared in many forms
 uses van Kampen diagrams to show that combable groups are finitely
presented and satisfy an exponential isoperimetric inequality. This argument
originated  in \cite{E+} (pages 52 and 152); cf.~\cite{mrb:geom} page 600.
 We record a version of it here because
it provides the template for the homotopies described in Section \ref{s:homotopies}.
We remind the reader that if  a word $w$ in the letters $\A$  represents the
identity in  $G=\<\A\mid R\>$, then
$\area(w)$ is the least integer $N$
such that $w$ can be expressed in the free group $F(\A)$ as a product of $N$ conjugates of 
the defining relators  and their inverses. 

\begin{prop}\label{p:fp} If $G$ admits $k$-Lipschitz contractions then $G=\<\A\mid R\>$
where  $R$ consists of all
words of length at most $2(k+1)$ that represent the identity in $G$.  Moreover, 
if $w=1$ in $G$, then
$\area(w)\le |w|. (|\A|+1)^{k|w|}$. \end{prop}

\begin{proof} We must prove that the 2-complex $X$ obtained from the Cayley graph $\mathcal C_\A(G)$
by attaching 2-cells along all loops of length at most $2(k+1)$ is simply connected. For this it is enough
to explain how to contract any edge-loop in $X$. Let $l:L\to X^{(1)}$ be an edge loop
labelled by $w\in\A^*$
(with $L=[0,|w|]\subset\mathbb R$).
Let
$S$ be the set of vertices in the image of $h$ and let $H_S:S\times\mathbb N\to G$  be as in the definition of
$k$-Lipschitz contractions. Let $N\in\N$ be the least integer such that $H_S(\ast,N)={\rm{id}}_S$.
We cellulate  $D=L\times [0,N]$ as a squared complex in
the obvious manner. Then $h_S(x,n):=H_S(l(x),n)$ is a map  from  the 0-skeleton of $D$ to $G$. 
Given a directed 1-cell in $D$ with initial vertex  $u$ and terminus $v$, we label it
 by a shortest word in $\A^*$ that equals  
$h(u)^{-1}h(v)$. Note that this word (which may be empty) has length at most $1$ if $u=(x,n)$ and
$v=(x,n+1)$, and length at most $k$ if $u=(x,n)$ and $v=(y,n)$. By construction, there is a 2-cell in
$X$ whose attaching map describes the loop in  $\mathcal C_\A(G)$ labelled by the word that one reads
around the boundary of each 2-cell in $D$. (Edges labelled by the empty word are collapsed as are 2-cells whose entire boundary is collapsed.)
Thus the map $D^{(1)}\to X^{(1)}$ that extends $h_S$
and is described by the labelling of $1$-cells, extends to a map from $D$ to $X$. This map gives a 
contraction of the original loop  $l$.

We define $\Lambda(w)$ to be the set of positive integers $M$ for which there is a map $h:(L\times [0,M])^{(0)} \to G$ with the following properties:

\begin{itemize} 
\item $\rho(h(u,n), h(u,n+1)) \le 1$ for all $u\in L$ and $n<M$

\item
$\rho(h(u,n), h(v,n))\le k$ for all adjacent vertices $u,v\in L$ 

\item
$h|_{L^{(0)}\times\{M\}}$ agrees with $l$ and $h|_{L^{(0)}\times\{0\}}$ is a constant map
\end{itemize}

The preceding argument shows that  $\Lambda(w)$ is non-empty. It also shows that $h$
 extends to a map $L\times [0,M]\to X$ that sends each open 2-cell of
$L\times [0,M]$ homeomorphically to an open 2-cell of $X$, or else collapses it.
The easier implication in van Kampen's Lemma 
(see \cite{mrb:bfs} p.49) implies that $\area(w)\le M\, |w|$ for all $M\in\Lambda(w)$ .
Thus the lemma will be proved if we can argue that 
$\min \{ M \mid M\in\Lambda(w)\}\le (|\A|+1)^{k|w|}$.
Since $\A=\A^{-1}$ generates $G$, the shortest words representing each of the word differences
$h(x,n)^{-1}h(y,n)$ may be taken to be positive (or empty). It follows that as $n$ varies there are
at most $(|\A|+1)^{k|w|}$ possibilities for the $|w|$-tuple of words labelling the $|w|$-tuple of edges 
in $L\times\{n\}$. If the $|w|$-tuple of words labelling $L\times\{n\}$ and $L\times\{n'\}$ coincide
for some $n<n'$, then we can delete $L\times [n,n']$ to obtain a map showing that $M-n'+n$ is in $\Lambda(w)$.
 So in particular, if $M$ is minimal then there is
no repetition, and hence $M<(|\A|+1)^{k|w|}$. 
\end{proof}
 
A recursive upper bound on the Dehn function of a finitely presented group leads in an
obvious way to a solution to the word problem.

\begin{corollary}\label{c:wp} 
If $G=\<\A\>$ admits $k$-Lipschitz contractions
and one can list the words in
the letters 
 $\A$ of length at most $2(k+1)$ that represent the identity in $G$,
then one can solve the word problem in $G$. 
\end{corollary}     
 
\begin{thm}\label{real}\label{t:real}
There exists an algorithm that takes as input the following data:
\begin{enumerate}\setcounter{enumi}{-1} 
\item a positive integer $d$; 
\item  a finite set of generators $\A$ for a group $G$;
\item a constant $k$ such that  $G$  admits $k$-Lipschitz contractions;
\item a list of the words in the letters $\A$ that are of length at most $2(k+1)$  and 
that equal $1$ in $G$;
\end{enumerate}
and which constructs as output a finite connected semi-combinatorial cell complex $K$ with $\pi_1K\cong G$ and 
$\pi_iK=0$ for $2\le i\le d$.
\end{thm}

\begin{corollary} There exists an algorithm that, given an integer $d$ and a
finite presentation of an automatic group $G$, will construct
an explicit model for the compact $d$-skeleton of a $K(G,1).$
\end{corollary}
 
\begin{proof}[Proof of Corollary] One implements the algorithm in Section 5.2 of
\cite{E+} to find the automatic
structure. This gives both an explicit fellow-traveller constant $k$ and a 
solution to the word problem. One uses the solution to the word problem
to list the words of length at most at most $2(k+1)$ in the letters $\A$
that equal $1$ in $G$.
Theorem \ref{real} then applies (cf. Proposition \ref{find} and
Lemma \ref{list}).
\end{proof}

\subsection{Template of the construction for $K$}\label{ss:doK}
Let $G$ and $d$ be as in Theorem \ref{real}.
The complex $K$ will have {\emf{vital $n$-cells}}, the larger collection
of {\emf{inflated $n$-cells}}, and {\emf{translation cells}}. The first
two types of cells form nested subcomplexes
$$K^{(0)}\subset K^{(1)}_v\subset K_I^{(1)}
 \subset K^{(2)}_v\subset K_I^{(2)}
 \subset \cdots
\subset K^{(d+1)}_v\subset K_I^{(d+1)}
$$
The translation cells up to dimension $n$  
form a subcomplex $T_n$,
and $K^{(n)}=K_I^{(n)}\cup T_n$. By definition $K=K^{(d+1)}$.
Let $p:\tilde K\to K$ be the universal covering.

The key properties of the construction are that, for each $n\le d$ :

\begin{enumerate}
\item[(i)] Each finite subcomplex of $p^{-1}(K_v^{(n)})\subset\tilde K$ 
is contractible in $p^{-1}(K_I^{(n+1)})$;
\item[(ii)] $K_I^{(n)}\cup T_{n+1}$ strong deformation retracts to $K_v^{(n)}$.
\end{enumerate}

The following property plays an important role in an induction on dimension
that we use to define  translation cells.

\begin{enumerate}
\item[(iii)] There is an algorithm that, given a finite subcomplex as in (i),
will construct an explicit contraction of it.
\end{enumerate}

The complex $K$ will have fundamental group $G$ and 
the construction will be entirely algorithmic. We claim that Theorem \ref{t:real} follows. Indeed, given a map of an $n$-sphere into $K$, with $2\le n\le d$, by simplicial approximation we may assume that the image lies in $K^{(n)}\subset K_I^{(n)}\cup T_{n +1}$, which contracts to $K_v^{(n)}$, and 
by (i) any $n$-sphere in $K_v^{(n)}$ is contractible in $K_I^{(n+1)}$.


\subsection{Sensible labellings}

We have a fixed generating set $\A$ for $G$.

We take a set of symbols in bijection with the 
freely reduced words
in the free group on $\A$ that have length at most $k^r$ (including the empty word);
we call these {\emf{magnitude $k^r$ labels}}. For $r'>r$ we make the
obvious  identification of the magnitude $k^r$ labels with the corresponding
subset of the magnitude $k^{r'}$ labels. These labels will be attached to the
oriented edges of the 1-skeleton of our complex, so that 
if a directed edge $e$ has label $w\neq\emptyset$, then\footnote{ $\bar{e}$ is  the edge $e$ with reversed orientation.}
$\bar{e}$ has label $w^{-1}$.
A labelling of the edges around the
boundary of a square 
is said to be {\emf{sensible}}
if the product of labels, read with consistent orientation, is equal to the
identity in our group $G$. (Here we evaluate the label as the corresponding product
of generators $a\in \A$, of course.) A {\emf{sensible labelling}} of magnitude $k^r$
on  $\D^n$ (the $n$-cube) is a labelling of its directed 1-cells by labels of  magnitude $k^r$ that is sensible on
each 2-dimensional face. Note that the restriction of a labelling of any magnitude to any
face of $\D^n$ is a sensible labelling of the same magnitude.

Two labellings of $\D^n$ are said to be \emf{equivalent}
if one is carried to the other by a symmetry of $\D^n$.

We highlight a trivial but important observation which explains why we articulated Corollary \ref{c:wp}. 

\begin{lem} If the word problem is solvable in $G$ then there is an
algorithm that, given $d\in\N$, will list the finitely many (equivalence classes of) sensible labellings of any given magnitude
for cubes up to dimension $d$
(and then halt).\qed
\end{lem}

\subsection{The complex $K$}\label{s:homotopies}

There is a single $0$-cell in $K$.

The 1-cells are in bijection with and are labelled
by reduced word over $\A$ of length
at most $k^d$ (including the empty word $\emptyset$). 
Choose (arbitrarily) an orientation on each edge.
Denote by $e_w$ the edge labelled by $w\in \A^*$.
Identify the edges $\overline{e}_w$ and $e_{w^{-1}}$
(where $w^{-1}$ is the inverse of $w$ in the free group on $A$, and 
$\overline{e}$ is the edge $e$ with opposite orientation).
Those edges labelled by words of length at most one will be called \emf{vital 1-cells}.

There is one 2-cell for each equivalence
class of labelled squares as in Figure \ref{fig:square}.
\begin{figure}[htb]
\footnotesize
\psfrag{0}{$\emptyset$}
\psfrag{a}{$a$}
\psfrag{b}{$b$}
\psfrag{c}{$c$}
\psfrag{d}{$d$}
\psfrag{0}{$\emptyset$}
\psfrag{*}{$*$}
\begin{center}
\includegraphics{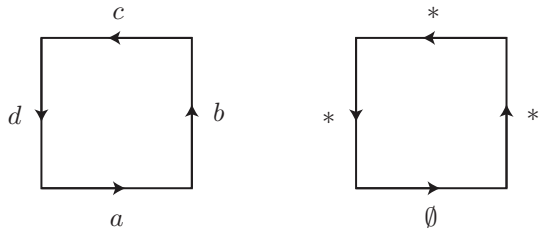}
\end{center}
\caption{\label{fig:square} 
Labelled squares. Edges are labelled by words
$a,b,c,d\in\{ w\in \A^*\| |w|\le k^d \}$ satisfying $abcd=1$ in $G$.
The square on the right maps via semi-combinatorial gluing maps,
the edges labelled $*$ mapping to the 0-cell of $K$.
}
\end{figure}

%
%
%
%

The complex $K_I^{(n)}$ is defined to
have one $n$-cell for each equivalence class of labellings of $\D^n$ of magnitude $k^d$. 
In addition $K_I$ has one {\emf{degenerate}} $n$-cell for each
equivalence class of  labellings of the 1-cells of $\D^n$ by elements of the set $\{\emptyset, *\}$, where  
at least one 1-cell of $\D^n$
is labelled by $\emptyset$.
The attaching map of a degenerate cell 
sends the directed 1-cells labelled $\emptyset$ to the directed 1-cell of $K$ labelled $\emptyset$, it  collapses the 1-cells labelled $*$ to the 0-cell of $K$, and it collapses any $k$-dimensional face whose entire 1-skeleton is labelled $*$.
The cells in $K_I^{(n)}$ are called {\emf{inflated}}.

For each positive integer $n\le d+1$,  the subcomplex $K_v^{(n)}\subset K_I^{(n)}$
consists of the $i$-cells, with $i\le n$, that have labellings of magnitude $k^{i-1}$
together with the  degenerate $i$-cells. The cells in $K_v^{(n)}$ are called {\emf{vital}}.
  
All we need to know about the translation cells for the moment is that there will be no
translation cells of dimension less than 2 and the translation 2-cells have attaching maps given by sensible labels (see Figure \ref{f:trans2}).

\begin{remark}\label{r:need-d} A noteworthy feature of the above construction is that it depends heavily on the integer $d$ fixed at the beginning of the
procedure. Since one knows that the $n$-skeleton of a complex $K$ with $d$-connected universal cover can serve as the $n$-skeleton of a classifying space
for $\pi_1K$, one would prefer an algorithmic construction of $K(G,1)^{(d)}$ that avoids this dependence. But the dependence on $d$
is difficult to avoid in an explicit construction. It
emerges from the fact that during a $k$-Lipschitz contraction, the diameter of the 1-skeleton of any $n$-cell can expand by a factor of $k$: crudely speaking, this means that one has to have 2-cells whose attaching maps cover all possibilities up to scale $k$ (cf.~Proposition \ref{p:fp}); one then has to contract the 2-skeleton within the 3-skeleton of the universal cover, and the
natural construction of this contraction requires 3-cells whose attaching maps are larger (in an appropriate sense) by a further factor of $k$, and so on.
\end{remark} 

\begin{lem}\label{l:pi1G} The maps $K_v\hookrightarrow K_I\hookrightarrow K$ induce
isomorphisms of fundamental groups, and $\pi_1K\cong G$.
\end{lem}

\begin{proof} The 2-skeleton of $K_v$ is obtained from that of the standard 2-complex $P$
of the presentation for $G$ given in Proposition \ref{p:fp} by adding an additional 1-cell
 labelled $\emptyset$, a 1-cell labelled by each word $w\in\A^*$ with $2\le |w|\le k^d$, and many extra 2-cells. The 1-cell labelled $\emptyset$ 
is null-homotopic in $K_I$ because we have the degenerate 2-cell with boundary label 
$(\ast,\ast,\ast,\emptyset)$. If $|w|\ge 2$, then $w\iso w_0a$ for some
$a\in\A$ and $|w_0|<|w|$. Thus the  edge labelled $w$ can be 
homotoped into the subcomplex with 1-cells labelled by words of
lesser length. An obvious induction on $|w|$ now implies that
$P\hookrightarrow K_v^{(2)}$
induces an epimorphism of fundamental groups. To see that this
is actually an isomorphism, 
 it suffices to note that the definition of {\em{sensible}}
is framed so that the additional 2-cells impose on the generators $\A$ only
relations that are valid in $G$.  The
 same considerations apply to  
 $P\hookrightarrow K_I^{(2)}\cup T_2=K^{(2)}.$  
 \end{proof}

With this lemma in hand, we can identify $G$ with the $0$-skeleton of $\tilde K$ and regard
the Cayley graph $C_\A(G)$ as a subcomplex of $\tilde K^{(1)}$. 
It is also justifies abbreviating  $p^{-1}(K_v^{(n)})$ to $\tilde K_v^{(n)}$ and
$p^{-1}(\tilde K_I^{(n)})$ to $\tilde K_I^{(n)}$.
We do so in the following proposition.

\begin{prop}\label{p:contract}
If $G$ admits $k$-Lipschitz contractions, then any finite subcomplex  $S\subset\tilde K_v^{(n)}$
is contractible in $\tilde K_I^{(n+1)}$.
\end{prop}

\begin{proof} We follow the proof of Proposition \ref{p:fp}.  Let $S_0$ be the vertex set of $S$,
let $H_{S_0}:S_0\times\mathbb N\to G=\tilde K^{(0)}$  be as in the definition of
$k$-Lipschitz contractions and let $N\in\N$ be the least integer such that $H_{S_0}(s,N)=s$
for all $s\in S_0$.
We cellulate  $Y=S\times [-1,N]$  in
the obvious manner: there are (horizontal) $m$-cells of the form $e\times [t,t+1]$, with
$e$ an $(m-1)$-cell of $S$, and (vertical) $m$-cells of the form $e'\times\{t\}$, with $e'$ an $m$-cell of $S$.
 The attaching maps of the cells
of $S$ determine the attaching maps of the cells in $Y$. 

We restrict $H_{S_0}$ to $S_0\times [0,N]$, then extend this to a map
$h_S:Y^{(0)}\to G=\tilde K^{(0)}$ by sending $S\times \{-1\}$ to the same vertex as $S\times \{0\}$.
If $E$ is an $m$-cell in $Y$ with attaching map $\phi_E:\D^m\to Y$ and $\varepsilon$
is a 1-cell of $\D^m$ whose endpoints map to $u$ and $v$, then we label $\varepsilon$ by
a shortest word in the letters $\A$ that equals 
$h_S(u)^{-1}h_S(v)$. The 1-cells mapping to $S_0\times [0,1)$ are labelled $\ast$ (where $\ast$ is the
special label introduced above for degenerate cells).

The two key points to observe are: first, since the attaching maps of the 
$m$-cells of $S$ send the 1-cells 
of   
$\D^m$
to edge-paths of length at most $k^n$, the labels we have assigned
to the 1-cells of   
cubical cells for
$Y$ 
are words of length at most $k^{n+1}$ (because
$H_{S_0}$ is a $k$-Lipschitz contraction); secondly, the labellings of  these cells of
$Y$ are
sensible, by construction. 

Since we added an $m$-cell to $K_I^{(m)}$ for each sensible labelling of $\D^m$ of magnitude
$k^{n+1}$, the labelling of the cells in $Y$ determines a natural map $Y\to K_I^{(n+1)}$ that extends
$h_S$, maps $(x,N)$ to $x$ for all $x\in S$, and is constant on $S\times\{-1\}$. Thus we
have constructed the desired contraction of $S$ in $K_I^{(n+1)}$.
\end{proof}
 
\subsection{Algorithmic construction of contractions and translation cells}

We follow  the second part of the proof of Proposition \ref{p:fp} 
to prove:

\begin{addendum} 
Given the data described in Theorem \ref{t:real}, there is an algorithm that will construct the contractions in Proposition \ref{p:contract}.
\end{addendum}

\begin{proof} Given $S$, one fixes a positive integer $N$ and tries to attach a label $h(u)\in G$
to each 0-cell in $S\times [-1,N]$, and a word $w_\e$ to the directed edges $\e$ of the domains of the
characteristic maps $\phi_E:\D^m\to S\times [-1,N]$ so that
\begin{itemize}
\item if $\phi(\e)$ joins $u$ to $v$ then $w_\e = h_S(u)^{-1}h_S(v)$ in $G$,

\item the labels $w_\e$ form a sensible labelling of magnitude $k^{n+1}$ on  $\D^m$
for each $m$-cell in $S\times [0,N]$,  

\item
the labelling of each cell mapping to $S\times\{N\}$ coincides with the labelling determining the characteristic map of the corresponding cell of $S\subset \tilde K^{(n)}$,

\item
the cells mapping to $S\times [-1,0)$ have  degenerate labellings, with the
1-cells mapping to  $S\times\{-1\}$ all labelled $\ast$.
\end{itemize}

In the preceding proposition we proved that for some $N$ such a choice of labels exists, so
one can algorithmically run all over all possible choices, picking labels arbitrarily
and using the solution to the word problem in $G$ (Corollary \ref{c:wp}) to check if the choices satisfy
the above conditions. (In fact, as in Proposition \ref{p:fp}, one has an {\em a priori} bound on $N$
that is an exponential function of the number of cells in $S$.)
\end{proof}
 
 We now define, inductively, the translation cells.

There is one translation 2-cell for each word $w\in F(\A)$ of length between 2 and $k^d$, that is, for each 1-cell in $K_I$ that is not a vital 1-cell.
Each translation 2-cell is of the form shown in Figure \ref{f:trans2} with attaching maps determined by the indicated (sensible) labelling.

\begin{figure}[htb]
\psfrag{0}{$\emptyset$}
\psfrag{1}{$a_1$}
\psfrag{2}{$a_2$}
\psfrag{n}{$a_n$}
\psfrag{d}{\ $\vdots$}
\psfrag{w}{$w$}
\begin{center}
\resizebox{4cm}{!}{\includegraphics{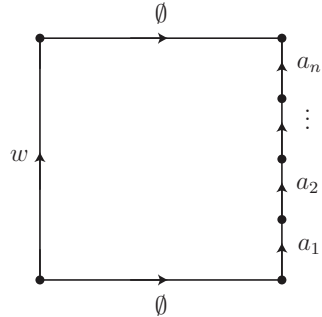}}
\end{center}
\caption{\label{f:trans2} 
A translation 2-cell. The labels on the edges are shown.
There is such a 2-cell for each word $w=a_1a_2\dots a_n\in F(\A)$ of length between 2 and $k^d$.
}
\end{figure}

Let $n\ge 1$ be an integer no greater than $d+1$.
There is one {\emf{translation cell}} of dimension $n+1$ for each $n$-cell of $K_I$ that is not vital.
 The attaching map of one codimension-1 face is the characteristic
map of the given inflated $n$-cell; write $\D^n = \D^{n-1}\times [0,1]$ and assume this 
is the face $\D^{n-1}\times\{0\}$. On the faces of the form $F\times [0,1]$ with $F<\D^{n-1}$
the attaching map is the translation cell corresponding to $F$ (which is well-defined, by induction).
We now have the attaching map defined on the boundary of $\D^{n-1}\times\{1\}$, yielding
a subcomplex of $K_v^{(n-1)}$ which is contractible in $K_I^{(n)}$
(again these are inductive assumptions). The addendum above yields an explicit contraction,
defined as a map from a cellulation of $\partial\D^{n-1}\times [0,1]$ to $K_I^{(n)}$. We identify this last complex
with a cylinder joining the boundary of  $\D^{n-1}\times\{1\}$ to a concentric $(n-1)$-cube
near the centre of  $\D^{n-1}\times\{1\}$, and we then complete the description of the
attaching map of our translation cell by sending the interior of this $(n-1)$ to the vertex
of $K_I$. Figure \ref{f:trans3} depicts the case of a translation 3-cell.

\begin{figure}[htb]
\footnotesize
\psfrag{0}{\tiny
$\emptyset$}
\psfrag{1}{$a_1$}
\psfrag{2}{$a_2$}
\psfrag{n}{$a_n$}
\psfrag{w}{$w$}
\psfrag{*}{\tiny$*$}
\begin{center}
\resizebox{6cm}{!}
{\includegraphics{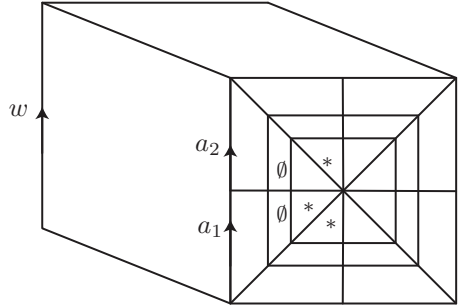}}
\end{center}
\caption{\label{f:trans3} 
A translation 3-cell. The back face is labelled by a single inflated 2-cell.
The side faces are translation 2-cells. The front face is subdivided into
conical wedges which are shown in Figure \ref{fig:conical}
}
\end{figure}

\begin{figure}[htb]
\footnotesize
\psfrag{0}{
$\emptyset$}
\psfrag{1}{$a_1$}
\psfrag{2}{$a_2$}
\psfrag{n}{$a_n$}
\psfrag{w}{$w$}
\psfrag{*}{$*$}
\begin{center}
\resizebox{0.9\textwidth}{!}
{\includegraphics{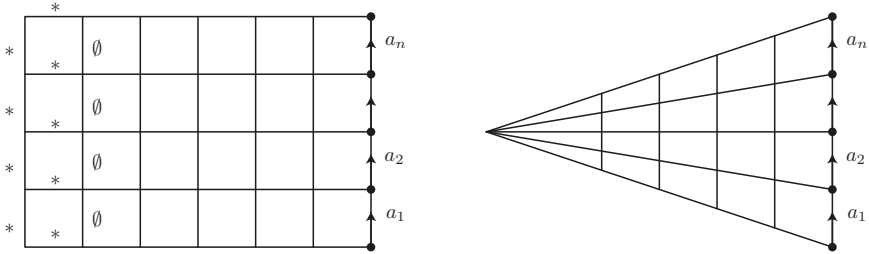}}
\end{center}
\caption{\label{fig:conical} 
The figure on the left shows a standard homotopy of  an edge path. The map from this to the complex $K$ factors through the conical quotient shown on the right.
}
\end{figure}

\begin{remark} It is important to note that this construction is entirely algorithmic.
\end{remark}
 
To complete the proof of Theorem \ref{t:real} it only remains to 
check that item (ii) of subsection 3.2 holds.

\begin{lem}\label{l:translate} 
$K_I^{(n)}\cup T_{n+1}$ strong deformation retracts to $K_v^{(n)}$. \end{lem}

\begin{proof} We have constructed a translation $(n+1)$-cell $t$
for each non-vital $n$-cell $e$ in $K_I$. The translation cell
provides a homotopy pushing $e$ into $K_v^{(n)}$, and by construction the  homotopies for distinct cells agree on faces of
intersection. 
\end{proof}

\subsection{Proof of Theorems B and C}

Theorem \refthmB\ follows immediately from Theorem \ref{t:real}
and the fact
that the Epstein algorithm calculates a fellow-traveller
constant of the automatic group from
any presentation of that group (cf. Proposition 1.2).
Elementary algebra allows one to compute the  (co)homology
of the complex $K$ from its cellular chain complex,  and this equals
the (co)homology of $\pi_1K$ in dimensions up to $d$, so
Theorem \refthmBb\ is an immediate consequence of 
Theorem \refthmB.
\qed 

\section{Proof of Theorem A}

In Section \ref{s:reduce} we reduced Theorem A 
to the problem of enumerating the central
extensions of a fixed biautomatic group
 $Q$ by a given finitely generated abelian group $Z$.
 So in the light of Theorem \refthmB, the following
 proposition completes the proof.

\begin{prop}\label{extenumerate}
If one has an explicit finite model for the 3-skeleton of a
$K(G,1)$, then one can irredundantly enumerate the central
extensions of $G$ with given finitely generated kernel
$A$.
\end{prop}

\begin{proof}
Elementary algebra allows one to explicitly calculate  
cellular 2-cocycles representing the elements of $H^2(G,A)$. Each
such cocycle $\sigma$ is an assignment of elements of
$A$ to the (oriented) 2-cells of $K$. We may assume that $K$ has
only one vertex (contracting a maximal tree if necessary), in
which case the 2-skeleton of $K$ corresponds to a (finite)
presentation for $G$ in the sense that $G=\langle S \mid
r_1,\dots,r_m\rangle$, where the $s_i\in S$ are the oriented
1-cells (there is a chosen orientation, so that $S$ is in
bijection with the physical 1-cells) and the $r_j$ are the
attaching paths of the 2-cells ---  there is a choice of an
oriented starting corner; different choices would lead to $r$
being replaced by a cyclic permutation of itself or its (free)
inverse. We write $|r|$ for the oriented 2-cell with boundary
label $r$.

The extension of $G$ by $A$ corresponding to $[\sigma]\in H^2(G,A)$
is the group with presentation 
$$
\langle
S\cup X \mid Y;
r_i=a_i, i=1,\dots,m \text{ and }  [x,y]=1\ \forall\ x\in X,
y\in S\rangle,
$$
where $A=\langle X\mid Y\rangle$ and $a_i$ is a word in $X^{\pm
1}$ which equals $\sigma(|r_i|)$ in $A$. Thus, from the 3-skeleton
of the $K(G,1)$ one obtains  a collection of representatives for
the elements of $H^2(G,A)$, and from that an irredundant
enumeration of the central extensions of $G$ by $A$, in the form
of finite presentations.

Note that if we had chosen a different starting corner for $r_i$
(but kept the orientation the same) then in this presentation we
would instead of $r_i=a_i$ have $r_i^*=a_i$, where
$r_i^*=s_jr_is_j^{-1}$ (freely) for some  $s_j$. Thus the
presentation obtained via this change would differ from the one
above by obvious Tietze transformations exploiting the relations
$[x,s_j]=1$. And a change of choice of orientation for $|r_i|$
would simply replace $r_i=a_i$ by $r_i^{-1}=a_i^{-1}$, since
$\sigma(|r_i^{-1}|)=-\sigma(|r_i|)$.
\end{proof}

\section{Proof of Theorem D}

Recall that the translation number of an element $g$ of a
finitely generated group is
$$
\tau(g)=
\lim_{m\to\infty} \frac{1}{m}d(1, g^m),
$$
where $d$ is a fixed word metric on the group. The translation
functions $g\mapsto \tau(g)$ associated to different word metrics
are Lipschitz equivalent, hence the statement `$g$ has non-zero
translation length' is independent of generating set. Also, since
$d(1,x^{-1}g^mx)$ and $d(1,g^m)$ differ by at most $2\,d(1,x)$,
the number $\tau(g)$ depends only on the conjugacy class of $g$.
And if $g$ and $h$ commute, then $\tau(gh)\le\tau(g)+\tau(h)$.

\begin{lem}\label{finitecentre}
Let $G$ be a finitely generated group in which central elements of
infinite order have positive translation numbers and in which the
torsion subgroup of the centre is finite (e.g., a biautomatic
group). Then $G/Z(G)$ is a group with finite centre. Moreover, if
$Z(G)$ is torsion-free,
then the centre of $G/Z(G)$ is trivial.
\end{lem}

\begin{proof}
Let $a$ be an element of $G$ that maps to a central element in
$G/Z(G)$. Then for all $g\in G$ we have $gag^{-1}=za$, where $z$
(which depends on $a$ and $g$) is central in $G$. But then
$g^nag^{-n}=z^na$ for all positive integers $n$. This implies that
the translation number of $z$ is zero, as for all $n\in\N$ we
have:
\begin{align*}
n\tau(z)&=\tau(z^n)=\tau(z^naa^{-1})\\
&\le \tau(z^na) +
 \tau(a^{-1})\text{ (since $z^na$ and $a^{-1}$ commute)}\\
&=\tau(g^nag^{-n})
+\tau(a) = 2\tau(a).
\end{align*}


Thus  for all $g\in G$, we have $a^{-1}ga = zg$, with $z$ a
torsion element of $Z(G)$. Since the torsion subgroup of $Z(G)$ is
finite and $G$ is finitely generated, it follows that there are
only finitely many possibilities for the inner automorphism
$\text{\rm{ad}(a)}\in\text{\rm{Inn}}(G) =G/Z(G)$. In other words,
$Z(G/Z(G))$ is finite.
\end{proof}

\begin{remark}
The finite presentations of biautomatic groups with finite centre,
and the finite presentations of biautomatic groups with
torsion-free centre, both form recursively enumerable classes.
\end{remark}

In the light of this lemma, and Theorem A, we obtain:

\begin{thmC}
The isomorphism problem among biautomatic groups is
solvable if and only if the isomorphism problem is solvable
among biautomatic groups with finite centre.
\end{thmC}

\begin{remark}
It also follows from Lemma \ref{finitecentre} that if one could
solve the isomorphism problem for biautomatic groups with trivial
centre, then one could solve the isomorphism problem for
torsion-free biautomatic groups. A similar reduction pertains for
biautomatic groups that are perfect, even in the presence  of
torsion.
\end{remark}

\section{Rational structures on abelian groups}

In Section \ref{s:1} we constructed an automatic structure for
the centre $Z(\Gamma)$, from this obtained a finite presentation
and then used the isomorphism problem for finitely generated
abelian groups to identify the isomorphism type of
$Z(\Gamma)$. In this section we present some results which
derive information directly from the automatic structure.

We remind the reader that a {\emf{rational structure}} for a 
group $G$ with finite semigroup generators $X$ is a regular
language $\L\subset X^*$ that maps bijectively to $G$
under the natural map $X^*\to G$.

\begin{lem}\label{growth}
Given a finitely generated abelian group $A$ and a rational
structure\footnote{no fellow-traveller
property is assumed}
 $\L\to A$, one can
decide the torsion free rank $\rho$ of $A$. Specifically, it is
the degree of growth of the regular language $\L$.
\end{lem}

\begin{proof}
Since words of length at most $n$ map into the ball of radius $n$
in $A$, the growth of $\L$ is no larger than that of $A$. So $\L$
is a regular language with polynomial growth of degree at most
$\rho$. In particular it is a union of basic languages of the form
$a_0l_1^*a_1\cdots l_n^*a_n$, with $n\le\rho$. As
$\L\to A$ is injective,
 each of the words $l_i$ projects to an infinite order
element in $A$. Therefore, the image of each of the finitely many
basic sublanguages is contained in a bounded neighbourhood of a
subgroup $\Z^n\leq A$, where $n\le\rho$. Since this image is the
whole of $A$, at least one of these subgroups must actually have
rank $n=\rho$. Thus the degree of polynomial growth of $\L$ is
$\rho$.
\end{proof}

It is possible to have the same regular language mapping
bijectively to different abelian groups (of the same torsion free
rank). For example, let $\L\subseteq \{x^{\pm 1},y\}^*$ be the
language defined by $\L=x^{\pm }*\cup x^{\pm *}y$. Then $\L$ maps bijectively
onto $\Z\cong\<a\|-\>$ via $x^{\pm 1} \mapsto a^{\pm 2} $,
$y\mapsto a$, and also onto $\Z\times
C_2\cong\<b,c\|[b,c]=1=c^2\>$ via $x^{\pm 1} \mapsto b^{\pm 1} $,
$y\mapsto c$.

\smallskip
Continuing with the notation of Lemma \ref{growth}, we have
an abelian group $A$ with  rational structure $\L\to A$
where $\L$ is a union
$\L=\L_1\cup\cdots\cup\L_r$ of languages of the form
$\L_i=a_0l_1^*a_1\cdots l_n^*a_n$.

\begin{lem}
Each sublanguage $\L_i$ contains at most one element which projects
to a finite order element of $A$.
\end{lem}

\begin{proof}
Suppose that $\gamma\in\mu(\L_i)$ has order $p$. Let
$f=\mu(a_0a_1\cdots a_n) $ and $g_i=\mu(l_i)$. If
$\gamma=fg_1^{m_1}\cdots g_n^{m_n}$, then
$\gamma^p=f^pg_1^{pm_1}\cdots g_n^{pm_n}=1$. For any
$\xi\in\mu(\L_i)$ not equal to $\gamma$, we have
$\xi=fg_1^{r_1}\cdots g_n^{r_n}$ with
$(r_1,\dots,r_n)\neq(m_1,\dots,m_n)$. Therefore
$\xi^p=f^pg_1^{pr_1}\cdots g_n^{pr_n}= g_1^{p(r_1-m_1)}\cdots
g_n^{p(r_n-m_n)}$ has infinite order.
\end{proof}
 



\begin{thebibliography}{NR}
\frenchspacing


\bibitem{alonso} J.~M. Alonso,
\emph{In\'egalit\'es isop\'erim\'etriques
 et quasi-isom\'etries},
 C.R.A.S. Paris S\'erie 1 {\bf 311} (1990), 761--764.


\bibitem{AB}
J. M. Alonso and M. R. Bridson, \emph{Semihyperbolic groups},
Proc. London Math Soc. (3) {\bf 70} (1995), 56--114.



\bibitem{mrb:geom}
M.~R.~Bridson,
{\em On the geometry of normal forms in discrete groups},
Proc. London Math. Soc. (3) {\bf 67} (1993),   596616.

\bibitem{mrb:isoC}
M.~R.~Bridson,
{\em The conjugacy and isomorphism problems for combable groups},
 Math. Ann. {\bf 327} (2003), 305314.



\bibitem{mrb:bfs}
M. R. Bridson, \emph{The geometry of the word problem},
in ``Invitations to Geometry and Topology" (M.R.~Bridson, S.M.~Salamon,
eds.), Oxford Univ. Press., Oxford, 2002. pp. 29--91.

\bibitem{mrb:icm}
M.~R. Bridson.
\newblock Non-positive curvature and complexity for finitely presented groups.
\newblock In {\em International {C}ongress of {M}athematicians. {V}ol. {II}},
  pages 961--987. Eur. Math. Soc., Z\"urich, 2006.



\bibitem{B}
K. S. Brown, \emph{Cohomology of groups},
Springer, New York, 1982.


\bibitem{DGroves} 
F. Dahmani and D. Groves.
\newblock The isomorphism problem for toral relatively hyperbolic groups.
\newblock {\em Publ. Math. Inst. Hautes \'Etudes Sci.}, (107):211--290, 2008.


\bibitem{DGuir} 
F. Dahmani and V. Guirardel.
\newblock The isomorphism problem for all hyperbolic groups. 
\newblock arXiv:1002.2590 


\bibitem{E+}
D.B.A. Epstein, J.W. Cannon, D.F. Holt, S.V.F. Levy,
M.S. Paterson, and W.P. Thurston, \emph{Word processing in
groups}, Jones and Bartlett Publishers, Boston, MA, 1992.

\bibitem{Ger}
S. M. Gersten,
\emph{Finiteness properties of asynchronously automatic groups},
in {\it Geometric group theory (Columbus, OH, 1992)}, 121--133, de Gruyter, Berlin, 1995.

\bibitem{GS}
S. M. Gersten\ and\ H. B. Short, \emph{Rational subgroups of biautomatic
groups}, Ann. of Math. (2) {\bf 134} (1991), no.~1, 125--158.


\bibitem{CG}
C.~McA.~Gordon, 
{\em Some embedding theorems and undecidability questions for
groups},
 Combinatorial and geometric group theory (Edinburgh, 1993),
105110, London Math. Soc. Lecture Note Ser., 204, Cambridge Univ.
Press, Cambridge, 1995.

\bibitem{M}
L. Mosher, \emph{Central quotients of biautomatic groups},
Comment. Math. Helv. {\bf 72} (1997), no.~1, 16--29.

\bibitem{NR}
W.~D. Neumann and L.~Reeves, \emph{Central extensions of word
hyperbolic groups}, Ann. of Math. (2) \textbf{145} (1997), no.~1, 183--192.

\bibitem{S}
Z.~Sela, \emph{The isomorphism problem for hyperbolic groups. {I}}, Ann. of
Math. (2) \textbf{141} (1995), no.~2, 217--283.

\end{thebibliography}
\end{document}